\documentclass[oneside,english]{amsart}
\usepackage[T1]{fontenc}
\usepackage[latin9]{inputenc}
\usepackage{tipa}
\usepackage{amsthm}

\makeatletter

\numberwithin{equation}{section}
\numberwithin{figure}{section}
\theoremstyle{plain}
\newtheorem{thm}{\protect\theoremname}
  \theoremstyle{plain}
  \newtheorem{lem}[thm]{\protect\lemmaname}
  \theoremstyle{plain}
  \newtheorem{prop}[thm]{\protect\propositionname}
  \theoremstyle{plain}
  \newtheorem{cor}[thm]{\protect\corollaryname}

\makeatother

\usepackage{babel}
  \providecommand{\corollaryname}{Corollary}
  \providecommand{\lemmaname}{Lemma}
  \providecommand{\propositionname}{Proposition}
\providecommand{\theoremname}{Theorem}

\begin{document}

\title[Improved Jensen-type inequalities via linear interpolation]{Improved Jensen-type inequalities via linear interpolation and applications
}

\author{Daeshik Choi}
\author{Mario Krni\' c}
\author{Josip Pe\v cari\' c}
\keywords{convex function, Jensen inequality, Young inequality, Kantorovich constant, Specht
ratio, arithmetic mean, geometric mean, Heinz mean.}

\subjclass[2010]{26D15, 26A51, 15A45.}

\address{Southern Illinois University, Edwardsville, Dept. of Mathematics
and Statistics, Box 1653, Edwardsville, IL 62026.}
\email{dchoi@siue.edu}
\address{University of Zagreb, Faculty of Electrical Engineering and Computing,  Unska 3, 10000 Zagreb, Croatia}
\email{mario.krnic@fer.hr}
\address{University of Zagreb, Faculty of Textile Technology,  Prilaz baruna Filipovi\' ca 28a, 10000 Zagreb, Croatia}
\email{pecaric@element.hr}
\begin{abstract}
In this paper we develop a general method for improving Jensen-type  inequalities for convex and, even more generally, for piecewise convex functions. Our main result relies
on the linear interpolation of a convex function. As a consequence, we obtain improvements of some recently established Young-type inequalities. Finally, our method is also applied to matrix case.
In such a way we obtain improvements of some important matrix inequalities known from the literature.
\end{abstract}

\maketitle

\section{Introduction}

The classical Young inequality, or the arithmetic-geometric mean inequality, states that
\begin{equation}
(1-v)a+vb\geq a^{1-v}b^{v},\label{eq:Young_ineq}
\end{equation}where  $a,b>0$ and $0\leq v\leq1$.
Refining this inequality and its reverse has taken the attention of
numerous researchers. Kittaneh and Manasrah \cite{Kittaneh_Manasrah},
improved (\ref{eq:Young_ineq}) to
\begin{equation}
(1-v)a+vb\geq a^{1-v}b^{v}+r_{0}(v)(\sqrt{a}-\sqrt{b})^{2},\label{eq:Kittaneh_Mana1}
\end{equation}
where $r_{0}(v)=\min\{v,1-v\}$. Moreover, Zhao and Wu \cite{Zhao_Wu}, established
even more accurate improvement:
\begin{eqnarray}
 &  & (1-v)a+vb\geq a^{1-v}b^{v}+r_{0}(v)(\sqrt{a}-\sqrt{b})^{2}\label{eq:zhao_Wu1}\\
 &  & \qquad+r_{1}(v)\left[\left(\sqrt{a}-\sqrt[4]{ab}\right)^{2}\chi_{(0,\frac{1}{2})}(v)+\left(\sqrt[4]{ab}-\sqrt{b}\right)^{2}\chi_{(\frac{1}{2},1)}(v)\right],\nonumber
\end{eqnarray}
where $r_{1}(v)=\min\{2r_{0}(v),1-2r_{0}(v)\}$ and $\chi_{I}(v)$
stands for the characteristic function of an interval $I$, defined by $\chi_{I}(v)=\begin{cases}
1, & v\in I\\
0, & v\notin I
\end{cases}$.

On the other hand, the reverses of inequalities (\ref{eq:Kittaneh_Mana1}) and (\ref{eq:zhao_Wu1})
read as follows \cite{Kitta_M2,Zhao_Wu}:
\begin{equation}
(1-v)a+vb\leq a^{1-v}b^{v}+R_{0}(v)(\sqrt{a}-\sqrt{b})^{2}\label{eq:Kitta_Ma_reverse}
\end{equation}
and
\begin{eqnarray}
 &  & (1-v)a+vb\leq a^{1-v}b^{v}+R_{0}(v)(\sqrt{a}-\sqrt{b})^{2}\label{eq:zhao_Wu_rev}\\
 &  & \qquad-r_{1}(v)\left[\left(\sqrt{b}-\sqrt[4]{ab}\right)^{2}\chi_{(0,\frac{1}{2})}(v)+\left(\sqrt[4]{ab}-\sqrt{a}\right)^{2}\chi_{(\frac{1}{2},1)}(v)\right],\nonumber
\end{eqnarray}
where $R_{0}(v)=1-r_{0}(v)$.

Other types of improvements of the Young inequality have been studied
in  numerous recent papers. For example, Wu and Zhao \cite{Wu_Zhao}, showed a pair of relations
\begin{eqnarray}
(1-v)a+vb & \geq & K_{1}(a,b)^{r_{1}(v)}a^{1-v}b^{v}+r_{0}(v)(\sqrt{a}-\sqrt{b})^{2},\label{eq:Wu}\\
(1-v)a+vb & \leq & K_{1}(a,b)^{-r_{1}(v)}a^{1-v}b^{v}+R_{0}(v)(\sqrt{a}-\sqrt{b})^{2},\nonumber
\end{eqnarray}
 where $K_{1}(a,b)=\frac{\left(\sqrt{a}+\sqrt{b}\right)^{2}}{4\sqrt{ab}}$.
Recently, Liao and Wu \cite{Liao_Wu}, have proven the inequalities
\begin{eqnarray}
 &  & (1-v)a+vb\geq K_{2}(a,b)^{r_{2}(v)}a^{1-v}b^{v}+r_{0}(v)(\sqrt{a}-\sqrt{b})^{2}\label{eq:Zhao_Wu-1}\\
 &  & \hfill\hphantom{K_{2}(a,b)^{r_{2}}a^{1-v}b^{v}}+r_{1}(v)\left[(\sqrt[4]{ab}-\sqrt{a})^{2}\chi_{(0,\frac{1}{2})}(v)+(\sqrt{b}-\sqrt[4]{ab})^{2}\chi_{(\frac{1}{2},1)}(v)\right],\nonumber \\
 &  & (1-v)a+vb\leq K_{2}(a,b)^{-r_{2}(v)}a^{1-v}b^{v}+R_{0}(v)(\sqrt{a}-\sqrt{b})^{2}\nonumber \\
 &  & \hfill\hphantom{K_{2}(a,b)^{r_{2}}a^{1-v}b^{v}}-r_{1}(v)\left[(\sqrt[4]{ab}-\sqrt{b})^{2}\chi_{(0,\frac{1}{2})}(v)+(\sqrt{a}-\sqrt[4]{ab})^{2}\chi_{(\frac{1}{2},1)}(v)\right],\nonumber
\end{eqnarray}
where $r_{2}(v)=\min\{2r_{1}(v),1-2r_{1}(v)\}$ and $K_{2}(a,b)=\frac{\left(\sqrt[4]{a}+\sqrt[4]{b}\right)^{2}}{4\sqrt[4]{ab}}$.
The constants of the form $\frac{(M+m)^{2}}{4Mm}$ are called Kantorovich
constants.

Further, Dragomir \cite{Dragomir}, showed the following pair of inequalities that hold for any
 $a,b>0$ and $0\leq v\leq1$:
\begin{eqnarray}
(1-v)a+vb & \geq & a^{1-v}b^{v}+\frac{1}{2}v(1-v)\left(\ln\frac{b}{a}\right)^{2}\min\{a,b\},\label{eq:Dragomir1}\\
(1-v)a+vb & \leq & a^{1-v}b^{v}+\frac{1}{2}v(1-v)\left(\ln\frac{b}{a}\right)^{2}\max\{a,b\}.\nonumber
\end{eqnarray}
Meanwhile, assuming $a,b\geq1$ and $0\leq v\leq1$, Minculete \cite{Mincluete},
proved that
\begin{eqnarray}
(1-v)a+vb & \geq & a^{1-v}b^{v}+r_{0}(v)(\sqrt{a}-\sqrt{b})^{2}+\alpha(v)(\ln\frac{b}{a})^{2},\label{eq:Mincule_1}\\
(1-v)a+vb & \leq & a^{1-v}b^{v}+R_{0}(v)(\sqrt{a}-\sqrt{b})^{2}+\alpha(v)\left(\ln\frac{b}{a}\right)^{2},\nonumber
\end{eqnarray}
where
\[
\alpha(v)=\frac{1}{2}v(1-v)-\frac{1}{4}r_{0}(v)=\frac{1}{4}r_{0}(v)|2v-1|.
\]
Finally, utilizing the Specht ratio
\[
S(t)=\frac{t^{1/(t-1)}}{e\ln t^{1/(t-1)}},
\]
Furuichi and Tominaga \cite{Furuichi,Tominaga}, showed that the series of inequalities
\begin{equation}
S\left(c^{r_{0}(v)}\right)a^{1-v}b^{v}\leq(1-v)a+vb\leq S(c)a^{1-v}b^{v},\label{eq:Specht_ineq}
\end{equation}
where $c=a^{-1}b$, holds  for any
$a,b>0$ and $0\leq v\leq1$.

Basically, the Young inequality (\ref{eq:Young_ineq}) is a consequence of the famous Jensen inequality
\begin{equation}
\label{jensen}
f((1-v)a+vb)\leq (1-v)f(a)+vf(b),
\end{equation}
where $f$ is a convex function defined on the interval $I$, $a,b\in I$, and $0\leq v\leq 1$. Clearly, the Young inequality (\ref{eq:Young_ineq}) follows from (\ref{jensen})
by putting $f(x)=-\ln x$, where $\ln$ stands for  a natural logarithm.

The main objective of this paper is to provide a unified treatment of Young-type inequalities presented in this Introduction. More precisely,
we will present a general improvement of a Jensen-type inequality related to piecewise convex functions and use it to refine some well-known classical inequalities.
As an application, we will also derive improved versions of some important matrix inequalities known from the literature. It should be noticed
here that the operator or matrix inequalities related to the
scalar inequalities introduced in this section can be found in some recent papers including \cite{Dragomir,Furuichi,Kittaneh_Manasrah,Liao_Wu,Mincluete,Wu_Zhao,Zhao_Wu}.

\section{The main result related to convex and piecewise convex functions}

In this section we give an improved version of the Jensen inequality that will, in some way, gather the relations presented in the previous section. Our main result will rely
on the linear interpolation of a convex function.

Throughout the paper, we will use the functions $r_{n}(v)$ defined
recursively by
\begin{eqnarray*}
r_{0}(v) & = & \min\{v,1-v\},\\
r_{n}(v) & = & \min\{2r_{n-1}(v),1-2r_{n-1}(v)\},
\end{eqnarray*}
for $0\leq v\leq1$. Note that $r_{0}(v)$ and $r_{1}(v)$ can be
rewritten as
\[
r_{0}(v)=\begin{cases}
v, & 0\leq v\leq\frac{1}{2}\\
1-v, & \frac{1}{2}<v\leq1
\end{cases},\quad r_{1}(v)=\begin{cases}
2v, & 0\leq v\leq\frac{1}{4}\\
1-2v, & \frac{1}{4}<v\leq\frac{1}{2}\\
2v-1, & \frac{1}{2}<v\leq\frac{3}{4}\\
2-2v, & \frac{3}{4}<v\leq1
\end{cases}.
\]
Generally, $r_{n}(v)$ can be expressed as multipart functions.
\begin{lem}
\label{lem:r_n} \cite{Choi_reverse} Let $n$ be
a nonnegative integer and $0\leq v\leq1$. If $\frac{k-1}{2^{n}}\leq v\leq\frac{k}{2^{n}}$
for $k=1,\ldots,2^{n}$, then
\[
r_{n}(v)=\begin{cases}
2^{n}v-k+1, & \frac{k-1}{2^{n}}\leq v\leq\frac{2k-1}{2^{n+1}}\\
k-2^{n}v, & \frac{2k-1}{2^{n+1}}<v\leq\frac{k}{2^{n}}
\end{cases}.
\]
\end{lem}
\begin{proof}
We prove it by induction on $n$. The case $n=0$ is obvious. Assume
$\frac{k-1}{2^{n+1}}\leq v\leq\frac{k-1}{2^{n+1}}+\frac{1}{2^{n+2}}$.
If $k=2m-1$ is odd, then $\frac{m-1}{2^{n}}\leq v\leq\frac{m-1}{2^{n}}+\frac{1}{2^{n+2}}<\frac{2m-1}{2^{n+1}}$
and $r_{n}(v)=2^{n}v-m+1$ by induction. Since $v\leq\frac{2k-1}{2^{n+2}}$,
\[
r_{n}(v)=2^{n}v-\frac{k-1}{2}\leq\frac{2k-1}{4}-\frac{k-1}{2}=\frac{1}{4}
\]
and
\[
r_{n+1}(v)=\min\{2r_{n}(v),1-2r_{n}(v)\}=2r_{n}(v)=2^{n+1}v-k+1.
\]
If $k=2m$ is even, then $\frac{2m-1}{2^{n+1}}\leq v\leq\frac{m}{2^{n}}-\frac{1}{2^{n+2}}<\frac{m}{2^{n}}$
and $r_{n}(v)=m-2^{n}v$ by induction. Since $v\leq\frac{2k-1}{2^{n+2}}$,
\[
r_{n}(v)=\frac{k}{2}-2^{n}v\geq\frac{k}{2}-\frac{2k-1}{4}=\frac{1}{4}
\]
and
\[
r_{n+1}(v)=1-2r_{n}(v)=2^{n+1}v-k+1.
\]
Using the same argument, we can show that if $\frac{k-1}{2^{n+1}}+\frac{1}{2^{n+2}}<v\leq\frac{k}{2^{n+1}}$,
then $r_{n+1}(v)=k-2^{n+1}v$. We omit the detailed proof.
\end{proof}
The functions $r_{n}$ can be used for linear interpolation as follows.
\begin{lem}
\label{lem:varphi_interp}Let $f$ be a function defined on $[0,1]$.
For a nonnegative integer $N$, define $\varphi_{N}(v)$ by
\begin{eqnarray*}
\varphi_{N}(v) & = & (1-v)f(0)+vf(1)-\sum_{n=0}^{N-1}r_{n}(v)\sum_{k=1}^{2^{n}}\Delta_{f}(n,k)\chi_{(\frac{k-1}{2^{n}},\frac{k}{2^{n}})}(v),
\end{eqnarray*}
where
\[
\Delta_{f}(n,k)=f(\frac{k-1}{2^{n}})+f(\frac{k}{2^{n}})-2f(\frac{2k-1}{2^{n+1}})
\]
and the summation is assumed to be zero if $N=0$. Then, $\varphi_{N}(v)$
is the linear interpolation of $f(v)$ at $v=k/2^{N}$, $k=0,1,\ldots,2^{N}$. \end{lem}
\begin{proof}
First we note that since $r_{n}(\frac{k}{2^{n}})=0$ for $0\leq k\leq2^{n}$,
the interval of the characteristic function may contain boundary points.
For example, $\chi_{(\frac{k-1}{2^{n}},\frac{k}{2^{n}})}$ can be
replaced by $\chi_{(\frac{k-1}{2^{n}},\frac{k}{2^{n}}]}$ or $\chi_{[\frac{k-1}{2^{n}},\frac{k}{2^{n}}]}$.
We will show that
\begin{equation}
\varphi_{N}(v)=(k-2^{N}v)f(\frac{k-1}{2^{N}})+(2^{N}v-k+1)f(\frac{k}{2^{N}})\label{eq:show_this}
\end{equation}
for $\frac{k-1}{2^{N}}\leq v\leq\frac{k}{2^{N}}$ and $k=1,\ldots,2^{N}$
by induction on $N$. It is obvious for $N=0$. Now, assume that (\ref{eq:show_this})
holds and let $\frac{m-1}{2^{N+1}}\leq v\leq\frac{m}{2^{N+1}}$ for
$m=1,\ldots,2^{N+1}$. If $m=2k-1$, then $\frac{k-1}{2^{N}}\leq v\leq\frac{2k-1}{2^{N+1}}<\frac{k}{2^{N}}$
and
\begin{eqnarray*}
\varphi_{N+1}(v) & = & \varphi_{N}(v)-r_{N}(v)\Delta_{f}(N,k)\\
 & = & (k-2^{N}v)f(\frac{k-1}{2^{N}})+(2^{N}v-k+1)f(\frac{k}{2^{N}})-(2^{N}v-k+1)\Delta_{f}(N,k)\\
 & = & (2k-2^{N+1}v-1)f(\frac{k-1}{2^{N}})+(2^{N+1}v-2k+2)f(\frac{2k-1}{2^{N+1}})\\
 & = & (m-2^{N+1}v)f(\frac{m-1}{2^{N+1}})+(2^{N+1}v-m+1)f(\frac{m}{2^{N+1}})
\end{eqnarray*}
by Lemma \ref{lem:r_n}. Similarly, if $m=2k$, then $\frac{k-1}{2^{N}}<\frac{2k-1}{2^{N+1}}\leq v\leq\frac{k}{2^{N}}$
and
\begin{eqnarray*}
\varphi_{N+1}(v) & = & \varphi_{N}(v)-r_{N}(v)\Delta_{f}(N,k)\\
 & = & (k-2^{N}v)f(\frac{k-1}{2^{N}})+(2^{N}v-k+1)f(\frac{k}{2^{N}})-(k-2^{N}v)\Delta_{f}(N,k)\\
 & = & (2k-2^{N+1}v)f(\frac{2k-1}{2^{N+1}})+(2^{N+1}v-2k+1)f(\frac{k}{2^{N}})\\
 & = & (m-2^{N+1}v)f(\frac{m-1}{2^{N+1}})+(2^{N+1}v-m+1)f(\frac{m}{2^{N+1}}).
\end{eqnarray*}
\end{proof}
From now on, any summation having $\sum_{n=0}^{N-1}$ will be assumed
to be zero for $N=0$ and $\Delta_{f}(n,k)$ defined in Lemma \ref{lem:varphi_interp}
will be used throughout the paper.

Now, we are ready to state and prove our main result.
The following theorem is based on a fact that a convex function can be estimated
by using the linear interpolations $\varphi_{N}(v)$ in Lemma \ref{lem:varphi_interp}. In fact, such estimation provides a refinement of the Jensen inequality for a convex function defined on the interval $[0,1]$.
\begin{thm}
\label{lem:convex_lemma}Let $N$ be a nonnegative integer. If $f(v)$
is convex on $[0,1]$, then
\begin{eqnarray}
 &  & (1-v)f(0)+vf(1)\geq f(v)\label{eq:f_convex_ineq}\\
 &  & \qquad+\sum_{n=0}^{N-1}r_{n}(v)\sum_{k=1}^{2^{n}}\Delta_{f}(n,k)\chi_{(\frac{k-1}{2^{n}},\frac{k}{2^{n}})}(v)\nonumber
\end{eqnarray}
and
\begin{eqnarray}
 &  & (1-v)f(0)+vf(1)\leq f(0)+f(1)-f(1-v)\label{eq:f_convex_ineq_reverse}\\
 &  & \qquad-\sum_{n=0}^{N-1}r_{n}(v)\sum_{k=1}^{2^{n}}\Delta_{f}(n,2^{n}-k+1)\chi_{(\frac{k-1}{2^{n}},\frac{k}{2^{n}})}(v).\nonumber
\end{eqnarray}
\end{thm}
\begin{proof}
By Lemma \ref{lem:varphi_interp}, we have $\varphi_{N}(v)\geq f(v)$
which represents (\ref{eq:f_convex_ineq}). Replacing $v$ by $1-v$
in (\ref{eq:f_convex_ineq}) and noting that $r_{n}(v)=r_{n}(1-v)$, we
have
\begin{eqnarray*}
(1-v)f(0)+vf(1) & \leq & f(0)+f(1)-f(1-v)\\
 &  & -\sum_{n=0}^{N-1}r_{n}(v)\sum_{k=1}^{2^{n}}\Delta_{f}(n,k)\chi_{(\frac{k-1}{2^{n}},\frac{k}{2^{n}})}(1-v).
\end{eqnarray*}
Now, replacing $k$ by $2^{n}-k+1$ in the inner summation and noting that
\[
\frac{k-1}{2^{n}}<1-v<\frac{k}{2^{n}}\iff1-\frac{k}{2^{n}}<v<1-\frac{k-1}{2^{n}},
\]
we obtain (\ref{eq:f_convex_ineq_reverse}) and the proof is completed.
\end{proof}
It should be noticed here that $\Delta_{f}\geq0$ in the previous theorem since $f$ is convex.
Therefore the inequality (\ref{eq:f_convex_ineq}) represents the refinement of the Jensen inequality for a convex function defined on the interval $[0,1]$.

Furthermore, it is important to emphasize  that Theorem \ref{lem:convex_lemma} can also  be applied to piecewise convex functions. For example,
if $f(v)$ is convex on $[0,\frac{1}{2}]$ and $[\frac{1}{2},1]$
and $f(\frac{1}{2})\leq\frac{1}{2}(f(0)+f(1))$, then $f(v)$ fulfills
the inequalities as in the theorem. More generally, if $f(v)$ is convex
on intervals of the form $[\frac{k-1}{2^{N}},\frac{k}{2^{N}}]$ and
$\Delta_{f}(N,k)\geq0$ for $1\leq k\leq2^{N}$, then the inequalities (\ref{eq:f_convex_ineq})
and (\ref{eq:f_convex_ineq_reverse}) are still valid for $f$.

\section{Improved versions of Young-type inequalities}

In this section, we will see how the Jensen-type inequalities from Theorem \ref{lem:convex_lemma} can
be used to improve Young-type inequalities.
The most general forms of (\ref{eq:Kittaneh_Mana1}), (\ref{eq:zhao_Wu1}),
(\ref{eq:Kitta_Ma_reverse}), and (\ref{eq:zhao_Wu_rev}) have been
 proved recently.
\begin{thm}
\label{thm:Young_multiple} \cite{Choi_reverse,Sabab_Choi_Young}
Let $a,b>0$, $0\leq v\leq1$, and $N$ be a nonnegative integer. Then,
\begin{eqnarray}
(1-v)a+vb & \geq & a^{1-v}b^{v}+\sum_{n=0}^{N-1}r_{n}(v)\sum_{k=1}^{2^{n}}g_{n,k}(a,b)\chi_{(\frac{k-1}{2^{n}},\frac{k}{2^{n}})}(v),\label{eq:multiple_Y}\\
 & = & a^{1-v}b^{v}+r_{0}(v)(\sqrt{a}-\sqrt{b})^{2}\nonumber \\
 &  & +\sum_{n=1}^{N-1}r_{n}(v)\sum_{k=1}^{2^{n}}g_{n,k}(a,b)\chi_{(\frac{k-1}{2^{n}},\frac{k}{2^{n}})}(v),\nonumber
\end{eqnarray}
and
\begin{eqnarray}
(1-v)a+vb & \leq & a+b-a^{v}b^{1-v}\label{eq:Young_reverse_multi}\\
 &  & -\sum_{n=0}^{N-1}r_{n}(v)\sum_{k=1}^{2^{n}}g_{n,k}(b,a)\chi_{(\frac{k-1}{2^{n}},\frac{k}{2^{n}})}(v)\nonumber \\
 & = & 2\sqrt{ab}-a^{v}b^{1-v}+R_{0}(v)(\sqrt{a}-\sqrt{b})^{2}\nonumber \\
 &  & -\sum_{n=1}^{N-1}r_{n}(v)\sum_{k=1}^{2^{n}}g_{n,k}(b,a)\chi_{(\frac{k-1}{2^{n}},\frac{k}{2^{n}})}(v)\nonumber \\
 & \leq & a^{1-v}b^{v}+R_{0}(v)(\sqrt{a}-\sqrt{b})^{2}\nonumber \\
 &  & -\sum_{n=1}^{N-1}r_{n}(v)\sum_{k=1}^{2^{n}}g_{n,k}(b,a)\chi_{(\frac{k-1}{2^{n}},\frac{k}{2^{n}})}(v),\nonumber
\end{eqnarray}
where $g_{n,k}(a,b)=\Delta_{f}(n,k)$ with $f(v)=a^{1-v}b^{v}$, i.e.,
\begin{eqnarray*}
g_{n,k}(a,b) & = & a^{1-(k-1)/2^{n}}b^{(k-1)/2^{n}}+a^{1-k/2^{n}}b^{k/2^{n}}\\
 &  & -2a^{1-(2k-1)/2^{n+1}}b^{(2k-1)/2^{n+1}}\\
 & = & \left(\sqrt{a^{1-(k-1)/2^{n}}b^{(k-1)/2^{n}}}-\sqrt{a^{1-k/2^{n}}b^{k/2^{n}}}\right)^{2}.
\end{eqnarray*}
\end{thm}

Note that the inequalities (\ref{eq:Kittaneh_Mana1}), (\ref{eq:zhao_Wu1}),
(\ref{eq:Kitta_Ma_reverse}), and (\ref{eq:zhao_Wu_rev}) follow directly from Theorem
\ref{thm:Young_multiple} for $N=1$ and $N=2$. The original proof of Theorem \ref{thm:Young_multiple}
 was rather lengthy,  here we give a simple and elegant proof based on our  Theorem
\ref{lem:convex_lemma}.
\begin{proof}
{[}Theorem \ref{thm:Young_multiple}{]} Since $f(v)=a^{1-v}b^{v}$
is convex on $[0,1]$, the inequality (\ref{eq:multiple_Y}) follows from (\ref{eq:f_convex_ineq}),
where we note that if $n=0$, then
\[
r_{n}(v)\sum_{k=1}^{2^{n}}\Delta_{f}(n,k)\chi_{(\frac{k-1}{2^{n}},\frac{k}{2^{n}})}(v)=r_{0}(v)(\sqrt{a}-\sqrt{b})^{2}.
\]
Further, utilizing (\ref{eq:f_convex_ineq_reverse}) we have
\begin{eqnarray*}
(1-v)a+vb & \leq & a+b-a^{v}b^{1-v}-\sum_{n=0}^{N-1}r_{n}(v)\sum_{k=1}^{2^{n}}g_{n,k}(b,a)\chi_{(\frac{k-1}{2^{n}},\frac{k}{2^{n}})}(v)\\
 & = & 2\sqrt{ab}-a^{v}b^{1-v}+R_{0}(v)(\sqrt{a}-\sqrt{b})^{2}\\
 &  & -\sum_{n=1}^{N-1}r_{n}(v)\sum_{k=1}^{2^{n}}g_{n,k}(b,a)\chi_{(\frac{k-1}{2^{n}},\frac{k}{2^{n}})}(v).
\end{eqnarray*}
Now, the second inequality in (\ref{eq:Young_reverse_multi}) follows by
the arithmetic-geometric mean inequality
$
2\sqrt{ab}\leq a^{v}b^{1-v}+a^{1-v}b^{v}.
$
\end{proof}

The inequalities (\ref{eq:Wu}) and (\ref{eq:Zhao_Wu-1}) involving  Kantorovich
constants can also be generalized in the following way.
\begin{thm}
\label{thm:Kantoro} \cite{Choi_reverse} Let $a,b>0$, $0\leq v\leq1$,
and $N$ be a nonnegative integer. Then
\begin{eqnarray}
(1-v)a+vb & \geq & K_{N}(a,b)^{r_{N}(v)}a^{1-v}b^{v}\label{eq:Kanto_Y}\\
 &  & +\sum_{n=0}^{N-1}r_{n}(v)\sum_{k=1}^{2^{n}}g_{n,k}(a,b)\chi_{(\frac{k-1}{2^{n}},\frac{k}{2^{n}})}(v)\nonumber \\
 & = & K_{N}(a,b)^{r_{N}(v)}a^{1-v}b^{v}+r_{0}(v)(\sqrt{a}-\sqrt{b})^{2}\nonumber \\
 &  & +\sum_{n=1}^{N-1}r_{n}(v)\sum_{k=1}^{2^{n}}g_{n,k}(a,b)\chi_{(\frac{k-1}{2^{n}},\frac{k}{2^{n}})}(v)\nonumber
\end{eqnarray}
and
\begin{eqnarray}
 &  & (1-v)a+vb\leq a+b-K_{N}(a,b)^{r_{N}(v)}a^{v}b^{1-v}\label{eq:Kanto_Y_reverse}\\
 &  & \hfill\hphantom{a+b-K_{N}(a,b)^{r_{N}(v)}}-\sum_{n=0}^{N-1}r_{n}(v)\sum_{k=1}^{2^{n}}g_{n,k}(b,a)\chi_{(\frac{k-1}{2^{n}},\frac{k}{2^{n}})}(v)\nonumber \\
 &  & \hphantom{(1-v)a+vb}=2\sqrt{ab}-K_{N}(a,b)^{r_{N}(v)}a^{v}b^{1-v}+R_{0}(v)(\sqrt{a}-\sqrt{b})^{2}\nonumber \\
 &  & \hfill\hphantom{a+b-K_{N}(a,b)^{r_{N}(v)}}-\sum_{n=1}^{N-1}r_{n}(v)\sum_{k=1}^{2^{n}}g_{n,k}(b,a)\chi_{(\frac{k-1}{2^{n}},\frac{k}{2^{n}})}(v)\nonumber \\
 &  & \hphantom{(1-v)a+vb}\leq K_{N}(a,b)^{-r_{N}(v)}a^{1-v}b^{v}+R_{0}(v)(\sqrt{a}-\sqrt{b})^{2}\nonumber \\
 &  & \hfill\hphantom{a+b-K_{N}(a,b)^{r_{N}(v)}}-\sum_{n=1}^{N-1}r_{n}(v)\sum_{k=1}^{2^{n}}g_{n,k}(b,a)\chi_{(\frac{k-1}{2^{n}},\frac{k}{2^{n}})}(v),\nonumber
\end{eqnarray}
where
\[
K_{N}(a,b)=\frac{\left(a^{1/2^{N}}+b^{1/2^{N}}\right)^{2}}{4(ab)^{1/2^{N}}}
\]
and $g_{n,k}$ is  defined in Theorem \ref{thm:Young_multiple}.
\end{thm}
The original proof of the above theorem can  also be simplified
by virtue of Theorem \ref{lem:convex_lemma}.
\begin{proof}
{[}Theorem \ref{thm:Kantoro}{]} Let $f(v)=K_{N}(a,b)^{r_{N}(v)}a^{1-v}b^{v}$.
Since $K_{N}(a,b)$ does not depend on variable  $v$ and $r_{N}(v)$ is a line
segment on each interval $I_{m}=[\frac{m-1}{2^{N+1}},\frac{m}{2^{N+1}}]$
for $1\leq m\leq2^{N+1}$, $f(v)$ is of the form $\alpha\beta^{v}$
on $I_{m}$ for some $\alpha,\beta>0$. Thus $f$ is convex on $I_{m}$
for $1\leq m\leq2^{N+1}$. Moreover, since $r_{N}(\frac{k}{2^{N}})=0$
for $0\leq k\leq2^{N}$, a direct computation shows that
\[
\Delta_{f}(n,k)=\begin{cases}
g_{n,k}(a,b), & 0\leq n<N\\
0, & n=N
\end{cases}.
\]
Although the function $f$ is not convex
on $[0,1]$, it is convex on intervals $I_{m}$. Moreover, since $\Delta_{f}(N,k)=0$,  Theorem \ref{lem:convex_lemma} can be applied
to function $f$. This yields  the inequality (\ref{eq:Kanto_Y}) and the first inequality in (\ref{eq:Kanto_Y_reverse}).
Finally, the second inequality in (\ref{eq:Kanto_Y_reverse}) follows
simply from the arithmetic-geometric mean inequality:
\[
2\sqrt{ab}\leq K_{N}(a,b)^{r_{N}(v)}a^{v}b^{1-v}+K_{N}(a,b)^{-r_{N}(v)}a^{1-v}b^{v}.
\]
\end{proof}

It should be noticed here that the inequalities (\ref{eq:Wu}) and (\ref{eq:Zhao_Wu-1}) are the special
cases of Theorem \ref{thm:Kantoro} with $N=1$ and $N=2$. In order to conclude our discussion regarding the previous theorem,
we show that the Kantorovich constants $K_{N}(a,b)$ appearing in Theorem \ref{thm:Kantoro} are the best possible.

\begin{prop}
Let $N$ be a nonnegative integer. If $\xi(a,b)$ is a nonnegative
function such that
\begin{equation}
(1-v)a+vb\geq\xi(a,b)^{r_{N}(v)}a^{1-v}b^{v}+\sum_{n=0}^{N-1}r_{n}(v)\sum_{k=1}^{2^{n}}g_{n,k}(a,b)\chi_{(\frac{k-1}{2^{n}},\frac{k}{2^{n}})}(v)\label{eq:xi_Kantoro}
\end{equation}
for $a,b>0$ and $0\leq v\leq1$, then $\xi(a,b)\leq K_{N}(a,b)$.\end{prop}
\begin{proof}
Let $f(v)=\xi(a,b)^{r_{N}(v)}a^{1-v}b^{v}$. Similarly to the proof of Theorem
\ref{thm:Kantoro}, we can show that $\Delta_{f}(n,k)=g_{n,k}(a,b)$
and that $f$ is convex on $I_{m}=[\frac{m-1}{2^{N+1}},\frac{m}{2^{N+1}}]$,
for $1\leq m\leq2^{N+1}$. Since
\[
(1-v)a+vb-\sum_{n=0}^{N-1}r_{n}(v)\sum_{k=1}^{2^{n}}g_{n,k}(a,b)\chi_{(\frac{k-1}{2^{n}},\frac{k}{2^{n}})}(v)
\]
is the linear interpolation of $f(v)$ at $v=k/2^{N}$, for $0\leq k\leq2^{N}$,
by Lemma \ref{lem:varphi_interp}, the inequality (\ref{eq:xi_Kantoro}) holds \textit{if
and only if $\Delta_{f}(N,k)\geq0$}, for $0\leq k\leq2^{N}$.

Now, let
$v_{k}=\frac{k}{2^{N}}$. Since $r_{N}(v_{k-1})=r_{N}(v_{k})=0$ and
$r_{N}(\frac{v_{k-1}+v_{k}}{2})=\frac{1}{2}$, the condition {$\Delta_{f}(N,k)\geq0$}
is equivalent to
\[
\xi(a,b)^{1/2}a^{1-(v_{k-1}+v_{k})/2}b^{(v_{k-1}+v_{k})/2}\leq\frac{1}{2}\left(a^{1-v_{k-1}}b^{v_{k-1}}+a^{1-v_{k}}b^{v_{k}}\right),
\]
that is,
\[
\xi(a,b)\leq\frac{\left(a^{1/2^{N}}+b^{1/2^{N}}\right)^{2}}{4(ab)^{1/2^{N}}}.
\]
Therefore we have $\xi(a,b)\leq K_{N}(a,b)$, $a,b>0$.
\end{proof}

The inequalities in (\ref{eq:Dragomir1}), due to Dragomir, can also be improved by virtue of Theorem \ref{lem:convex_lemma}.
\begin{thm}
Let $a,b>0$ , $0\leq v\leq1$, and let $N$ be a nonnegative integer. Then,
\begin{eqnarray}
 &  & (1-v)a+vb\geq a^{1-v}b^{v}+\sum_{n=0}^{N-1}r_{n}(v)\sum_{k=1}^{2^{n}}g_{n,k}(a,b)\chi_{(\frac{k-1}{2^{n}},\frac{k}{2^{n}})}(v)\label{eq:Dra_multi1}\\
 &  & \hfill\hphantom{aaaaaaaaaaaaaaaa}+\left(\frac{v(1-v)}{2}-\sum_{n=0}^{N-1}\frac{r_{n}(v)}{2^{n+2}}\right)(\ln\frac{b}{a})^{2}\min\{a,b\}\nonumber
\end{eqnarray}
and
\begin{eqnarray}
 &  & (1-v)a+vb\leq a+b-a^{v}b^{1-v}-\sum_{n=0}^{N-1}r_{n}(v)\sum_{k=1}^{2^{n}}g_{n,k}(b,a)\chi_{(\frac{k-1}{2^{n}},\frac{k}{2^{n}})}(v)\label{eq:Dra_multi1_rev}\\
 &  & \hfill\hphantom{aaaaaaaaaaaaaaaa}-\left(\frac{v(1-v)}{2}-\sum_{n=0}^{N-1}\frac{r_{n}(v)}{2^{n+2}}\right)\left(\ln\frac{b}{a}\right)^{2}\min\{a,b\},\nonumber
\end{eqnarray}
where the function $g_{n,k}$ is  defined in Theorem \ref{thm:Young_multiple}.\end{thm}
\begin{proof}
Putting $f(v)=a^{1-v}b^{v}+\frac{1}{2}v(1-v)\left(\ln\frac{b}{a}\right)^{2}\min\{a,b\}$, we have

\[
f^{\prime\prime}(v)=(\ln\frac{b}{a})^{2}\left(a^{1-v}b^{v}-\min\{a,b\}\right)\geq0,
\]
so the inequalities (\ref{eq:Dra_multi1}) and (\ref{eq:Dra_multi1_rev}) follows directly
from Theorem \ref{lem:convex_lemma}, since
\[
\Delta_{f}(n,k)=g_{n,k}(a,b)-\frac{1}{2^{2n+2}}\left(\ln\frac{b}{a}\right)^{2}\min\{a,b\}.
\]

\end{proof}
Note that the inequality (\ref{eq:Dragomir1}) follows from the above theorem for
$N=0$. It is very interesting to compare relations (\ref{eq:multiple_Y}) and (\ref{eq:Dra_multi1}).
It can be shown that if $N\geq2$, then
\[
\frac{v(1-v)}{2}\leq\sum_{n=0}^{N-1}\frac{r_{n}(v)}{2^{n+2}},
\]
for $0\leq v\leq1$. Thus, the inequality (\ref{eq:Dra_multi1}) is weaker than (\ref{eq:multiple_Y})
for $N\geq2$. On the other hand,  in the case when $N=1$, the relation (\ref{eq:Dra_multi1}) is stronger than (\ref{eq:multiple_Y}), since
\[
\frac{1}{2}v(1-v)-\frac{1}{4}r_{0}(v)=\frac{1}{4}r_{0}(v)|1-2v|\geq0.
\]
Similarly, the inequality
(\ref{eq:Dra_multi1_rev}) is stronger than the first inequality in
(\ref{eq:Young_reverse_multi}) when $N=1$, and we have the following
result.
\begin{cor}
\label{cor:Dra_better_N1}Let $a,b>0$ and $0\leq v\leq1$. Then,
\begin{equation}
(1-v)a+vb\geq a^{1-v}b^{v}+r_{0}(v)(\sqrt{a}-\sqrt{b})^{2}+\alpha(v)\zeta(a,b)\label{eq:Mincu_better1}
\end{equation}
and
\begin{eqnarray}
(1-v)a+vb & \leq & a+b-a^{v}b^{1-v}-r_{0}(v)(\sqrt{a}-\sqrt{b})^{2}-\alpha(v)\zeta(a,b)\label{eq:Mincul_better2}\\
 & \leq & a^{1-v}b^{v}+R_{0}(v)(\sqrt{a}-\sqrt{b})^{2}-\alpha(v)\zeta(a,b),\nonumber
\end{eqnarray}
where
\begin{eqnarray*}
\alpha(v) & = & \frac{1}{2}v(1-v)-\frac{1}{4}r_{0}(v)=\frac{1}{4}r_{0}(v)|1-2v|,\\
\zeta(a,b) & = & (\ln\frac{b}{a})^{2}\min\{a,b\}.
\end{eqnarray*}
Moreover, (\ref{eq:Mincu_better1}) and the first inequality in (\ref{eq:Mincul_better2})
are stronger than the corresponding ones in (\ref{eq:Mincule_1})
for $a,b\geq1$. \end{cor}
\begin{proof}
The relations (\ref{eq:Mincu_better1}) and (\ref{eq:Mincul_better2}) follow
from (\ref{eq:Dra_multi1}) and (\ref{eq:Dra_multi1_rev}) with $N=1$
respectively, where the second inequality in (\ref{eq:Mincul_better2})
follows from the arithmetic-geometric mean inequality $2\sqrt{ab}\leq a^{v}b^{1-v}+a^{1-v}b^{v}$.

Now, assume that $a,b\geq1$. Since $\min\{a,b\}\geq1$, it is obvious that
(\ref{eq:Mincu_better1}) is stronger than the first inequality in
(\ref{eq:Mincule_1}). Moreover, from (\ref{eq:Mincul_better2}) we
have
\begin{eqnarray*}
(1-v)a+vb & \leq & a+b-a^{v}b^{1-v}-r_{0}(v)(\sqrt{a}-\sqrt{b})^{2}-\alpha(v)\left(\ln\frac{b}{a}\right)^{2}\\
 & = & 2\sqrt{ab}+R_{0}(v)(\sqrt{a}-\sqrt{b})^{2}-a^{v}b^{1-v}-\alpha(v)\left(\ln\frac{b}{a}\right)^{2}\\
 & = & a^{1-v}b^{v}+R_{0}(v)(\sqrt{a}-\sqrt{b})^{2}+\alpha(v)\left(\ln\frac{b}{a}\right)^{2}\\
 &  & +2\sqrt{ab}-a^{v}b^{1-v}-a^{1-v}b^{v}-\left(v(1-v)-\frac{1}{4}\right)\left(\ln\frac{b}{a}\right)^{2}.
\end{eqnarray*}
Thus, it suffices to show the relation
\[
2\sqrt{ab}\leq a^{1-v}b^{v}+a^{v}b^{1-v}+\left(v(1-v)-\frac{1}{4}\right)\left(\ln\frac{b}{a}\right)^{2}
\]
for $a,b\geq1$ and $0\leq v\leq1$. Denoting the right-hand side of the above inequality by
$f(v)$, we have
\[
f^{\prime\prime}(v)=\left(\ln\frac{b}{a}\right)^{2}\left(a^{1-v}b^{v}+a^{v}b^{1-v}-2\right).
\]
Since $a^{1-v}b^{v}+a^{v}b^{1-v}\geq2\sqrt{ab}\geq2$, it follows that $f$ is
convex. Moreover, since $f(v)=f(1-v)$, $f$ attains
its minimum value at $v=\frac{1}{2}$, that is, $f(v)\geq f(\frac{1}{2})=2\sqrt{ab}$.
\end{proof}

Now, our aim is to improve the series of inequalities in (\ref{eq:Specht_ineq}) which includes the Specht ratio.
Note that the Specht ratio $S(t)=t^{1/(t-1)}/(e\ln t^{1/(t-1)})$
has the following properties (see e.g. \cite{Tominaga}):
\begin{itemize}
\item $S(1)=\lim_{t\to1}S(t)=1$ and $S(t)=S(t^{-1})$ for $t>0$.
\item $S^{\prime}(t)<0$ for $0<t<1$ and $S^{\prime}(t)>0$ for $t>1$.
\end{itemize}
Before the corresponding improvement, we first give an auxiliary result regarding the Specht ratio.
\begin{lem}
\label{lem:Dt}Let $S(t)$ be the Specht ratio and define $D(t)$
by
\[
D(t)=\frac{1}{2}(t+t^{-1}),
\]
for $t>0$. Then,
\begin{enumerate}
\item $S(t)\leq D(t)$ for $t>0$,
\item For any $c>0$, $f(v)=D(c^{r_{0}(v)})c^{v}$ is convex on $[0,\frac{1}{2}]$
and $[\frac{1}{2},1]$. Moreover, $f(\frac{1}{2})=\frac{1}{2}(f(0)+f(1))$.
\end{enumerate}
\end{lem}
\begin{proof}
Since $S(t^{-1})=S(t)$ and $D(t^{-1})=D(t)$, for $t>0$, we will
show $S(t)\leq D(t)$ for $t>1$. Taking a natural logarithm, we
can show that $S(t)\leq D(t)\iff\psi(t)\geq0$, where
\[
\psi(t)=\ln(t^{2}+1)-\ln(2t)-\ln(t-1)-\frac{\ln t}{t-1}+1+\ln\ln t.
\]
A direct computation yields
\begin{eqnarray*}
\psi^{\prime}(t) & = & \frac{\ln t}{(t-1)^{2}}+\frac{1}{t\ln t}-\frac{2(t+1)}{(t^{2}+1)(t-1)}\\
 & \geq & \frac{2}{\sqrt{t}(t-1)}-\frac{2(t+1)}{(t^{2}+1)(t-1)}\\
 & = & 2\frac{t\sqrt{t}-1}{(t+\sqrt{t})(t^{2}+1)}>0,
\end{eqnarray*}
for $t>1$. Since $\lim_{t\to1+}\psi(t)=0$, it follows that $\psi(t)\geq0$ for $t\geq1$.

The convexity of $f$ is obvious, since
\[
f(v)=\begin{cases}
\frac{1}{2}(c^{2v}+1), & 0\leq v\leq\frac{1}{2}\\
\frac{1}{2}(c+c^{2v-1}), & \frac{1}{2}<v\leq1
\end{cases}.
\]
Finally, $f(\frac{1}{2})=D(\sqrt{c})\sqrt{c}=\frac{1}{2}(1+c)=\frac{1}{2}(f(0)+f(1))$.
\end{proof}

Now, the following improvement of the series of inequalities in (\ref{eq:Specht_ineq}) is also based on our Theorem \ref{lem:convex_lemma}.

\begin{thm}
\label{thm:Specht_ineq}Let $a,b>0$ and $0\leq v\leq1$. If  $N$ is a nonnegative integer, then
\begin{eqnarray}
 &  & (1-v)a+vb\geq S(c^{r_{0}(v)})a^{1-v}b^{v}\label{eq:Spetcht_better}\\
 &  & \hfill\hphantom{aaaaaaaaa}+\sum_{n=0}^{N-1}r_{n}(v)\sum_{k=1}^{2^{n}}\Delta_{f}(n,k)\chi_{(\frac{k-1}{2^{n}},\frac{k}{2^{n}})}(v)\nonumber
\end{eqnarray}
and
\begin{eqnarray}
 &  & (1-v)a+vb\leq a+b-a^{v}b^{1-v}S(c^{r_{0}(v)})\label{eq:Spetcht_rev_better}\\
 &  & \hfill\hphantom{aaaaaaaaa}-\sum_{n=0}^{N-1}r_{n}(v)\sum_{k=1}^{2^{n}}\Delta_{f}(n,2^{n}-k+1)\chi_{(\frac{k-1}{2^{n}},\frac{k}{2^{n}})}(v),\nonumber
\end{eqnarray}
where $c=a^{-1}b$ and $f(v)=S(c^{r_{0}(v)})a^{1-v}b^{v}$.\end{thm}
\begin{proof}
For $c>0$, let $f_{c}(v)=c^{v}S(c^{r_{0}(v)})$. We will show that
$f_{c}$ is convex on $[0,\frac{1}{2}]$ and $[\frac{1}{2},1]$. Since
\[
f_{-c}(v)=c^{-v}S(c^{r_{0}(v)})=c^{-1}c^{1-v}S(c^{r_{0}(1-v)})=c^{-1}f_{c}(1-v),
\]
we may assume  $c>1$ and show that $g(v)\equiv e(\ln c)f_{c}(v)$
is convex on $[0,\frac{1}{2}]$ and $[\frac{1}{2},1]$. From now on,
we will write any function $\alpha(v)$ simply as $\alpha$, for a convenience.
Letting $x=r_{0}/(c^{r_{0}}-1)$, $h=c^{x}/x$, and $g=c^{v}h$, a
straightforward computation yields
\begin{eqnarray*}
h^{\prime} & = & x^{\prime}\left(\ln c-\frac{1}{x}\right)h,\\
h^{\prime\prime} & = & x^{\prime\prime}\left(\ln c-\frac{1}{x}\right)h+\left(\frac{x^{\prime}}{x}\right)^{2}h+x^{\prime}\left(\ln c-\frac{1}{x}\right)h^{\prime}\\
 & = & h\left[x^{\prime\prime}\left(\ln c-\frac{1}{x}\right)+\left(\frac{x^{\prime}}{x}\right)^{2}+(x^{\prime})^{2}\left(\ln c-\frac{1}{x}\right)^{2}\right]
\end{eqnarray*}
and
\begin{eqnarray*}
g^{\prime} & = & c^{v}\left(h\ln c+h^{\prime}\right),\\
g^{\prime\prime} & = & c^{v}\left(h(\ln c)^{2}+2h^{\prime}\ln c+h^{\prime\prime}\right)\\
 & = & c^{v}h\left(\left[\ln c+x^{\prime}(\ln c-\frac{1}{x})\right]^{2}+\frac{1}{x^{2}}\left[xx^{\prime\prime}\left(x\ln c-1\right)+(x^{\prime})^{2}\right]\right).
\end{eqnarray*}
Thus, it suffices to show that
\begin{equation}
xx^{\prime\prime}\left(x\ln c-1\right)+(x^{\prime})^{2}\geq0.\label{eq:show_x_ineq}
\end{equation}
Since
\begin{eqnarray*}
x^{\prime} & = & r_{0}^{\prime}\left(\frac{1}{c^{r_{0}}-1}-\frac{r_{0}c^{r_{0}}\ln c}{(c^{r_{0}}-1)^{2}}\right)=\pm\frac{x}{r_{0}}\left(1-xc^{r_{0}}\ln c\right),\\
x^{\prime\prime} & = & \pm\frac{\left[x^{\prime}\left(1-2xc^{r_{0}}\ln c\right)-x^{2}c^{r_{0}}r_{0}^{\prime}(\ln c)^{2}\right]r_{0}-x\left(1-xc^{r_{0}}\ln c\right)r_{0}^{\prime}}{r_{0}^{2}}\\
 & = & \frac{x^{2}}{r_{0}^{2}}\left((2xc^{r_{0}}-r_{0})\ln c-2\right)c^{r_{0}}\ln c,
\end{eqnarray*}
the relation (\ref{eq:show_x_ineq}) can be rewritten as
\[
\left((2xc^{r_{0}}-r_{0})\ln c-2\right)\left(x\ln c-1\right)xc^{r_{0}}\ln c+\left(1-xc^{r_{0}}\ln c\right)^{2}\geq0.
\]
Replacing $x$ by $r_{0}/(c^{r_{0}}-1)$ and denoting $c^{r_{0}}$
by $t$, the above inequality reads
\[
\left(\frac{t+1}{t-1}\ln t-2\right)\left(\frac{\ln t}{t-1}-1\right)\frac{t\ln t}{t-1}+\left(1-\frac{t\ln t}{t-1}\right)^{2}\geq0
\]
for $t>1$. Multiplying by $(t-1)^{3}$ and letting $s=\ln t$, the
above expression becomes
\[
\left((t+1)s-2(t-1)\right)\left(s-t+1\right)ts+(t-1)\left(t-1-ts\right)^{2}\geq0,
\]
or equivalently,
\[
\xi(t)\equiv t^{3}+\left(s^{3}-3s^{2}-3\right)t^{2}+(s^{3}+3s^{2}+3)t-1\geq0
\]
for $t>1$. A straightforward computation shows
\begin{eqnarray*}
\xi_{1}=\xi^{\prime} & = & 3t^{2}+(2s^{3}-3s^{2}-6s-6)t+s^{3}+6s^{2}+6s+3,\\
\xi_{2}=t\xi_{1}^{\prime} & = & 6t^{2}+(2s^{3}+3s^{2}-12s-12)t+3s^{2}+12s+6,\\
\xi_{3}=t\xi_{2}^{\prime} & = & 12t^{2}+(2s^{3}+9s^{2}-6s-24)t+6s+12,\\
\xi_{4}=t\xi_{3}^{\prime} & = & 24t^{2}+(2s^{3}+15s^{2}+12s-30)t+6,\\
\xi_{5}=t\xi_{4}^{\prime} & = & 48t+2s^{3}+21s^{2}+42s-18.
\end{eqnarray*}
Since $t>1$ and $s>0$, it follows that $\xi_{5}>0$. Thus, $\xi(t)\geq0$ results
from
\[
\xi_{4}(1)=\cdots=\xi_{1}(1)=\xi(1)=0.
\]
We have shown that $f_{c}$ is convex on $[0,\frac{1}{2}]$ and $[\frac{1}{2},1]$.
Now, by Lemma \ref{lem:Dt} it follows that $\frac{1}{2}(\sqrt{c}+\sqrt{c}^{-1})\geq S(\sqrt{c})$
which is equivalent to $f(0)+f(1)\geq2f(\frac{1}{2})$. Thus, $f_{c}$
satisfies
\begin{eqnarray*}
1-v+vc & \geq & c^{v}S(c^{r_{0}(v)})-\sum_{n=0}^{N-1}r_{n}(v)\sum_{k=1}^{2^{n}}\Delta_{f_{c}}(n,k)\chi_{(\frac{k-1}{2^{n}},\frac{k}{2^{n}})}(v),\\
1-v+vc & \leq & 1+c-c^{1-v}S(c^{r_{0}(v)})-\sum_{n=0}^{N-1}r_{n}(v)\sum_{k=1}^{2^{n}}\Delta_{f_{c}}(n,2^{n}-k+1)\chi_{(\frac{k-1}{2^{n}},\frac{k}{2^{n}})}(v),
\end{eqnarray*}
by Theorem \ref{lem:convex_lemma}. Finally, letting $c=a^{-1}b$, we obtain (\ref{eq:Spetcht_better})
and (\ref{eq:Spetcht_rev_better}).
\end{proof}

In order to conclude this section, we give yet another improvement of the Young inequality, based on Theorem \ref{lem:convex_lemma} and Lemma \ref{lem:Dt}.
\begin{thm}
\label{thm:D_thm}Let $a,b>0$ and let $N$ be a nonnegative integer.
Define $g_{a,b}(v)$ by
$
g_{a,b}(v)=a^{1-2v}b^{2v},
$
 $0\leq v\leq1$.
\begin{enumerate}
\item If $0\leq v\leq\frac{1}{2},$ then
\begin{eqnarray*}
(1-v)a+vb & \geq & \frac{1}{2}(a^{1-2v}b^{2v}+a)\\
 &  & +\frac{1}{2}\sum_{n=1}^{N-1}r_{n}(v)\sum_{k=1}^{2^{n-1}}\Delta_{g_{a,b}}(n,k)\chi_{(\frac{k-1}{2^{n}},\frac{k}{2^{n}})}(v),\\
(1-v)a+vb & \leq & a+\frac{1}{2}b-\frac{1}{2}a^{2v}b^{1-2v}\\
 &  & -\frac{1}{2}\sum_{n=1}^{N-1}r_{n}(v)\sum_{k=1}^{2^{n-1}}\Delta_{g_{b,a}}(n,k)\chi_{(\frac{k-1}{2^{n}},\frac{k}{2^{n}})}(v).
\end{eqnarray*}

\item If $\frac{1}{2}<v\leq1$, then
\begin{eqnarray*}
(1-v)a+vb & \geq & \frac{1}{2}(a^{2-2v}b^{2v-1}+b)\\
 &  & +\frac{1}{2}\sum_{n=1}^{N-1}r_{n}(v)\sum_{k=1}^{2^{n-1}}\Delta_{g_{a,b}}(n,k)\chi_{(\frac{k-1}{2^{n}},\frac{k}{2^{n}})}(v-\frac{1}{2}),\\
(1-v)a+vb & \leq & \frac{1}{2}a+b-\frac{1}{2}a^{2v-1}b^{2-2v}\\
 &  & -\frac{1}{2}\sum_{n=1}^{N-1}r_{n}(v)\sum_{k=1}^{2^{n-1}}\Delta_{g_{b,a}}(n,k).\chi_{(\frac{k-1}{2^{n}},\frac{k}{2^{n}})}(v-\frac{1}{2}).
\end{eqnarray*}

\end{enumerate}
\end{thm}
\begin{proof}
Utilizing Theorem \ref{lem:convex_lemma} and Lemma \ref{lem:Dt} with
\begin{eqnarray*}
f(v) & = & D(c^{r_{0}(v)})a^{1-v}b^{v}\\
 & = & \frac{1}{2}(a^{-r_{0}(v)+1-v}b^{r_{0}(v)+v}+a^{r_{0}(v)+1-v}b^{-r_{0}(v)+v})\\
 & = & \begin{cases}
\frac{1}{2}(a^{1-2v}b^{2v}+a), & 0\leq v\leq\frac{1}{2}\\
\frac{1}{2}(a^{2-2v}b^{2v-1}+b), & \frac{1}{2}<v\leq1
\end{cases},
\end{eqnarray*}
we have
\begin{eqnarray}
(1-v)a+vb & \geq & f(v)\label{eq:D_v1}\\
 &  & +\sum_{n=1}^{N-1}r_{n}(v)\sum_{k=1}^{2^{n}}\Delta_{f}(n,k)\chi_{(\frac{k-1}{2^{n}},\frac{k}{2^{n}})}(v),\nonumber \\
(1-v)a+vb & \leq & a+b-f(1-v)\nonumber \\
 &  & -\sum_{n=1}^{N-1}r_{n}(v)\sum_{k=1}^{2^{n}}\Delta_{f}(n,2^{n}-k+1)\chi_{(\frac{k-1}{2^{n}},\frac{k}{2^{n}})}(v).\nonumber
\end{eqnarray}
Note  that the outer summation starts at $n=1$, since $\Delta_{f}(0,1)=0$.

If $0\leq v\leq\frac{1}{2},$ then  $f(v)=\frac{1}{2}(a^{1-2v}b^{2v}+a)$
and $f(1-v)=\frac{1}{2}(a^{2v}b^{1-2v}+b)$. Further, taking into account (\ref{eq:D_v1}), we have
\begin{eqnarray*}
(1-v)a+vb & \geq & \frac{1}{2}(a^{1-2v}b^{2v}+a)\\
 &  & +\sum_{n=1}^{N-1}r_{n}(v)\sum_{k=1}^{2^{n-1}}\Delta_{f}(n,k)\chi_{(\frac{k-1}{2^{n}},\frac{k}{2^{n}})}(v),\\
(1-v)a+vb & \leq & a+\frac{1}{2}b-\frac{1}{2}a^{2v}b^{1-2v}\\
 &  & -\sum_{n=1}^{N-1}r_{n}(v)\sum_{k=1}^{2^{n-1}}\Delta_{f}(n,2^{n}-k+1)\chi_{(\frac{k-1}{2^{n}},\frac{k}{2^{n}})}(v).
\end{eqnarray*}
Finally, since $1\leq k\leq2^{n-1}$, it follows that
\begin{eqnarray*}
\Delta_{f}(n,k) & = & \frac{1}{2}\Delta_{g_{a,b}}(n,k),\\
\Delta_{f}(n,2^{n}-k+1) & = & f(1-\frac{k}{2^{n}})+f(1-\frac{k-1}{2^{n}})-2f(1-\frac{2k-1}{2^{n+1}})\\
 & = & \frac{1}{2}\Delta_{g_{b,a}}(n,k).
\end{eqnarray*}

On the other hand, if $\frac{1}{2}<v\leq1$, then $f(v)=\frac{1}{2}(a^{2-2v}b^{2v-1}+b)$
and $f(1-v)=\frac{1}{2}(a^{2v-1}b^{2-2v}+a)$. Thus, utilizing (\ref{eq:D_v1}) we have,
\begin{eqnarray*}
(1-v)a+vb & \geq & \frac{1}{2}(a^{2-2v}b^{2v-1}+b)\\
 &  & +\sum_{n=1}^{N-1}r_{n}(v)\sum_{k=2^{n-1}+1}^{2^{n}}\Delta_{f}(n,k)\chi_{(\frac{k-1}{2^{n}},\frac{k}{2^{n}})}(v)\\
 & = & \frac{1}{2}(a^{2-2v}b^{2v-1}+b)\\
 &  & +\sum_{n=1}^{N-1}r_{n}(v)\sum_{k=1}^{2^{n-1}}\Delta_{f}(n,k+2^{n-1})\chi_{(\frac{k-1}{2^{n}},\frac{k}{2^{n}})}(v-\frac{1}{2})
\end{eqnarray*}
and
\begin{eqnarray*}
(1-v)a+vb & \leq & \frac{1}{2}a+b-\frac{1}{2}a^{2v-1}b^{2-2v}\\
 &  & -\sum_{n=1}^{N-1}r_{n}(v)\sum_{k=2^{n-1}+1}^{2^{n}}\Delta_{f}(n,2^{n}-k+1)\chi_{(\frac{k-1}{2^{n}},\frac{k}{2^{n}})}(v)\\
 & = & \frac{1}{2}a+b-\frac{1}{2}a^{2v-1}b^{2-2v}\\
 &  & -\sum_{n=1}^{N-1}r_{n}(v)\sum_{k=1}^{2^{n-1}}\Delta_{f}(n,2^{n-1}-k+1)\chi_{(\frac{k-1}{2^{n}},\frac{k}{2^{n}})}(v-\frac{1}{2}).
\end{eqnarray*}
Finally, if $1\leq k\leq2^{n-1}$, we have
\begin{eqnarray*}
\Delta_{f}(n,k+2^{n-1}) & = & f(\frac{1}{2}+\frac{k-1}{2^{n}})+f(\frac{1}{2}+\frac{k}{2^{n}})-2f(\frac{1}{2}+\frac{2k-1}{2^{n+1}})\\
 & = & \frac{1}{2}\Delta_{g_{a,b}}(n,k)
\end{eqnarray*}
and
\begin{eqnarray*}
\Delta_{f}(n,2^{n-1}-k+1) & = & f(\frac{1}{2}-\frac{k}{2^{n}})+f(\frac{1}{2}-\frac{k-1}{2^{n}})-2f(\frac{1}{2}-\frac{2k-1}{2^{n+1}})\\
 & = & \frac{1}{2}\Delta_{g_{b,a}}(n,k),
\end{eqnarray*}
which completes the proof.
\end{proof}

\section{Applications to some matrix inequalities}

Our aim in this section is to discuss some matrix inequalities that correspond to scalar inequalities derived in the previous section.

Throughout this section, we will use $M_{n}$ for the set of $n\times n$
complex matrices, $M_{n}^{+}$ for the subset of $M_{n}$ consisting
of positive definite matrices, and $|||\cdot|||$ for any unitarily
invariant norm. For $A\in M_{n}$, $A>0$ ($A\geq0$) means that
$A$ is positive definite (semidefinite). For Hermitian matrices $A,B\in M_{n}$,
$A<B$ $(A\leq B)$ implies that $B-A$ is positive definite (semidefinite).
The absolute value of $A\in M_{n}$ will be defined by $|A|=(A^{*}A)^{1/2}$.

For $A,B\in M_{n}^{+}$ and $0\leq v\leq1$, the $v$-weighted arithmetic
mean and geometric mean of $A$ and $B$ are defined, respectively,
by
\begin{eqnarray*}
A\nabla_{v}B & = & (1-v)A+vB,\\
A\sharp_{v}B & = & A^{1/2}(A^{-1/2}BA^{-1/2})^{v}A^{1/2}.
\end{eqnarray*}
For convenience of notation, we use $A\nabla B$ for $A\nabla_{\frac{1}{2}}B$
and $A\sharp B$ for $A\sharp_{\frac{1}{2}}B$.

In order to obtain matrix inequalities from the corresponding scalar inequalities,
we will use the operator monotonicity of continuous functions, that
is, if $f$ is a real valued continuous function defined on the spectrum
of a self-adjoint operator $A$, then $f(t)\geq0$ for every $t$
in the spectrum of $A$ implies that $f(A)$ is a positive operator.

Matrix inequalities that correspond to Theorems \ref{thm:Young_multiple} and \ref{thm:Kantoro} have been already established in papers \cite{Choi_reverse,Sabab_Choi_Young}. Now,
we are going to discuss matrix inequalities that correspond to Corollary \ref{cor:Dra_better_N1}, closely connected to some recent matrix inequalities due to Dragomir.

In order to do this, we will first generalize the definition of the geometric mean $A\sharp_{v}B=A^{1/2}(A^{-1/2}BA^{-1/2})^{v}A^{1/2}$.
Let $f$ be a continuous function defined on an interval $I$ containing
the spectrum of $A^{-1/2}BA^{-1/2}$. Then, using the functional calculus
for continuous functions, we define $A\sharp_{f}B$ by
\[
A\sharp_{f}B=A^{1/2}f(A^{-1/2}BA^{-1/2})A^{1/2}.
\]
Utilizing the scalar relation (\ref{eq:Dragomir1}), Dragomir \cite{Dragomir}, established the following series of inequalities
\begin{equation}
\frac{1}{2}v(1-v)A\sharp_{f_{\min}}B\leq A\nabla_{v}B-A\sharp_{v}B\leq\frac{1}{2}v(1-v)A\sharp_{f_{\max}}B,\label{eq:Dra_op}
\end{equation}
where $A,B\in M_{n}^{+}$, $0\leq v\leq1$, and
\begin{eqnarray*}
f_{\min}(x) & = & \min\{1,x\}(\ln x)^{2},\\
f_{\max}(x) & = & \max\{1,x\}(\ln x)^{2},
\end{eqnarray*}
where $x>0$. Now, by virtue of our Corollary \ref{cor:Dra_better_N1} we can obtain more accurate relations than those in (\ref{eq:Dra_op}).

\begin{thm}\label{prvi}
Let  $A,B\in M_{n}^{+}$ and $0\leq v\leq1$. Then,
\[
A\nabla_{v}B\geq A\sharp_{v}B+r_{0}(v)(A+B-2A\sharp B)+\alpha(v)A\sharp_{f_{\min}}B
\]
and
\begin{eqnarray*}
A\nabla_{v}B & \leq & A+B-A\sharp_{1-v}B-r_{0}(v)(A+B-2A\sharp B)-\alpha(v)A\sharp_{f_{\min}}B,\\
 & = & A\sharp B-A\sharp_{1-v}B+R_{0}(v)(A+B-2A\sharp B)-\alpha(v)A\sharp_{f_{\min}}B,
\end{eqnarray*}
where $\alpha(v)=\frac{1}{2}v(1-v)-\frac{1}{4}r_{0}(v)$ and $f_{\min}(x)=\min\{1,x\}(\ln x)^{2}$.\end{thm}
\begin{proof}
By Corollary \ref{cor:Dra_better_N1}, we have
\begin{eqnarray*}
 &  & 1-v+vc\geq c^{v}+r_{0}(v)(c+1-2\sqrt{c})+\alpha(v)f_{\min}(c),\\
 &  & 1-v+vc\leq1+c-c^{1-v}-r_{0}(v)(c+1-2\sqrt{c})-\alpha(v)f_{\min}(c)\\
 &  & \hphantom{1-v+vc}=2\sqrt{c}-c^{1-v}+R_{0}(v)(c+1-2\sqrt{c})-\alpha(v)f_{\min}(c)
\end{eqnarray*}
for $c>0$ and $0\leq v\leq1$. Now, substituting $c$ by $A^{-1/2}BA^{-1/2}$
and multiplying each inequality by $A^{1/2}$ both-sidedly, which preserves operator order, we obtain desired relations.
\end{proof}

Next, we give the matrix interpretation of  Theorem \ref{thm:D_thm}.

\begin{thm}\label{drugi}
Let $A,B\in M_{n}^{+}$ and $0\leq v\leq1$. Define $G_{n,k}(A,B)$
by
\[
G_{n,k}(A,B)=A\sharp_{(k-1)/2^{n-1}}B+A\sharp_{k/2^{n-1}}B-2A\sharp_{(2k-1)/2^{n}}B.
\]

\begin{enumerate}
\item If $0\leq v\leq\frac{1}{2}$, then
\begin{eqnarray*}
(1-v)A+vB & \geq & \frac{1}{2}(A\sharp_{2v}B+A)\\
 &  & +\frac{1}{2}\sum_{n=1}^{N-1}r_{n}(v)\sum_{k=1}^{2^{n-1}}G_{n,k}(A,B)\chi_{(\frac{k-1}{2^{n}},\frac{k}{2^{n}})}(v),\\
(1-v)A+vB & \leq & A+\frac{1}{2}B-\frac{1}{2}A\sharp_{1-2v}B\\
 &  & -\frac{1}{2}\sum_{n=1}^{N-1}r_{n}(v)\sum_{k=1}^{2^{n-1}}G_{n,k}(B,A)\chi_{(\frac{k-1}{2^{n}},\frac{k}{2^{n}})}(v).
\end{eqnarray*}

\item If $\frac{1}{2}<v\leq1$, then
\begin{eqnarray*}
(1-v)A+vB & \geq & \frac{1}{2}(A\sharp_{2v-1}B+B)\\
 &  & +\frac{1}{2}\sum_{n=1}^{N-1}r_{n}(v)\sum_{k=1}^{2^{n-1}}G_{n,k}(A,B)\chi_{(\frac{k-1}{2^{n}},\frac{k}{2^{n}})}(v-\frac{1}{2}),\\
(1-v)A+vB & \leq & \frac{1}{2}A+B-\frac{1}{2}A\sharp_{2-2v}B\\
 &  & -\frac{1}{2}\sum_{n=1}^{N-1}r_{n}(v)\sum_{k=1}^{2^{n-1}}G_{n,k}(B,A)\chi_{(\frac{k-1}{2^{n}},\frac{k}{2^{n}})}(v-\frac{1}{2}).
\end{eqnarray*}

\end{enumerate}
\end{thm}
\begin{proof}
Let $c>0$. Taking into account Theorem \ref{thm:D_thm} with $0\leq v\leq\frac{1}{2}$, we have
\begin{eqnarray*}
(1-v)+vc & \geq & \frac{1}{2}(c^{2v}+1)+\frac{1}{2}\sum_{n=1}^{N-1}r_{n}(v)\sum_{k=1}^{2^{n-1}}\Delta_{g_{1,c}}(n,k)\chi_{(\frac{k-1}{2^{n}},\frac{k}{2^{n}})}(v),\\
(1-v)+vc & \leq & 1+\frac{1}{2}c-\frac{1}{2}c^{1-2v}-\frac{1}{2}\sum_{n=1}^{N-1}r_{n}(v)\sum_{k=1}^{2^{n-1}}\Delta_{g_{c,1}}(n,k)\chi_{(\frac{k-1}{2^{n}},\frac{k}{2^{n}})}(v),
\end{eqnarray*}
where
\begin{eqnarray*}
\Delta g_{1,c}(n,k) & = & c^{(k-1)/2^{n-1}}+c^{k/2^{n-1}}-2c^{(2k-1)/2^{n}},\\
\Delta g_{c,1}(n,k) & = & c^{1-(k-1)/2^{n-1}}+c^{1-k/2^{n-1}}-2c^{1-(2k-1)/2^{n}}.
\end{eqnarray*}
Now,  the desired inequalities follow by substituting $c$ by $A^{-1/2}BA^{-1/2}$
and multiplying  each inequality by $A^{1/2}$ both-sidedly. The same conclusion can be drawn for the case $\frac{1}{2}<v\leq1$. We omit the detailed proof.
\end{proof}

The rest of this section will be dedicated to improving some important matrix inequalities known from the literature.
First, we deal with Heinz-type inequalities.
For $0\leq v\leq1$, the Heinz mean in parameter $v$ is defined by
\[
H_{v}(a,b)=\frac{a^{1-v}b^{v}+a^{v}b^{1-v}}{2},\qquad a,b>0.
\]
The Heinz mean is convex on $[0,1]$, as a function of variable $v$  and attains
its minimum value at $v=1/2$. Thus, the Heinz mean interpolates between the geometric mean
and the arithmetic mean, that is,
\[
\sqrt{ab}\leq H_{v}(a,b)\leq\frac{a+b}{2}.
\]
Similarly, it is easy to see that for any $A,B\in M_{n}^{+}$ holds relation
\begin{equation}
A\sharp B\leq H_{v}(A,B)\leq A\nabla B,\label{eq:Heinz_Arith}
\end{equation}
where
\[
H_{v}(A,B)=\frac{A\sharp_{v}B+A\sharp_{1-v}B}{2}.
\]
Now, by virtue of Theorem \ref{lem:convex_lemma}, we can improve the second inequality in (\ref{eq:Heinz_Arith}).
\begin{thm}\label{treci}
Let $A,B\in M_{n}^{+}$. If $N$ is a nonnegative integer, then
\begin{equation*}
\begin{split}
&H_{v}(A,B)\leq A\nabla B\\
 &\qquad -\sum_{n=0}^{N-1}r_{n}(v)\sum_{k=1}^{2^{n}}\left(H_{\frac{k-1}{2^{n}}}(A,B)+H_{\frac{k}{2^{n}}}(A,B)-2H_{\frac{2k-1}{2^{n+1}}}(A,B)\right)\chi_{(\frac{k-1}{2^{n}},\frac{k}{2^{n}})}(v).
\end{split}
\end{equation*}
\end{thm}
\begin{proof}
Let $c>0$. Since $f(v)=H_{v}(1,c)=(c^{v}+c^{1-v})/2$ is convex on
$[0,1]$, we have
\[
f(v)\leq f(0)-\sum_{n=0}^{N-1}r_{n}(v)\sum_{k=1}^{2^{n}}\Delta_{f}(n,k)\chi_{(\frac{k-1}{2^{n}},\frac{k}{2^{n}})}(v)
\]
by Theorem \ref{lem:convex_lemma}. By the functional calculus, we can replace $c$ by $A^{-1/2}BA^{-1/2}$.
Then, multiplying the obtained inequality by $A^{1/2}$ both-sidedly, we obtain the desired inequality.
\end{proof}
Note that the second inequality in (\ref{eq:Heinz_Arith}) follows
from the above theorem with $N=0$. Moreover, if $N=1$, we have
\[
H_{v}(A,B)\leq(1-2r_{0}(v))A\nabla B+2r_{0}(v)A\sharp B
\]
for all $0\leq v\leq1$, which was proved in \cite{Kittaneh_Krnic}.

Kittaneh \cite{Kittaneh}, showed that if
 $A,B\in M_{n}^{+}$, $X\in M_{n}$, and $0\leq v\leq1$, then
\begin{eqnarray}
|||A^{1-v}XB^{v}+A^{v}XB^{1-v}||| & \leq & 4r_{0}(v)|||A^{1/2}XB^{1/2}|||\label{eq:Kittaneh_1}\\
 &  & +(1-2r_{0}(v))|||AX+XB|||.\nonumber
\end{eqnarray}
This Heinz-type inequality for unitarily invariant norms can be improved as follows.
\begin{thm}
\label{thm:Kittaneh_ext}Let $A,B\in M_{n}^{+}$ and $X\in M_{n}$.
If $0\leq v\leq1$, then
\begin{eqnarray*}
 &  & |||A^{1-v}XB^{v}+A^{v}XB^{1-v}|||\leq|||AX+XB|||\\
 &  & \hphantom{aaaaaaaaa}-\sum_{n=0}^{N-1}r_{n}(v)\sum_{k=1}^{2^{n}}\Delta_{f}(n,k)\chi_{(\frac{k-1}{2^{n}},\frac{k}{2^{n}})}(v),
\end{eqnarray*}
where $f(v)=|||A^{1-v}XB^{v}+A^{v}XB^{1-v}|||$. \end{thm}
\begin{proof}
It follows from  Theorem \ref{lem:convex_lemma} since the function $f(v)=|||A^{1-v}XB^{v}+A^{v}XB^{1-v}|||$ is convex on $[0,1]$
(for more details, see \cite[Corollary IX.4.10]{Bhatia}).
\end{proof}
\noindent Considering the above theorem for $N=0$, we obtain the well-known Heinz inequality
\[
|||A^{1-v}XB^{v}+A^{v}XB^{1-v}|||\leq|||AX+XB|||,
\]
while for $N=1$, we have
\[
|||A^{1-v}XB^{v}+A^{v}XB^{1-v}|||\leq|||AX+XB|||-2r_{0}(v)\big(|||AX+XB|||-2|||A^{1/2}XB^{1/2}|||\big)
\]
which is simply (\ref{eq:Kittaneh_1}).

Now, consider the following relation that interpolates the matrix Cauchy-Schwartz inequality \cite[Corollary 4.31]{Zhan}:
\begin{eqnarray*}
|||\,|A^{1/2}XB^{1/2}|^{t}|||^{2} & \leq & |||\,|A^{1-v}XB^{v}|^{t}|||\cdot|||\,|A^{v}XB^{1-v}|^{t}|||\\
 & \leq & |||\,|AX|^{t}|||\cdot|||\,|XB|^{t}|||,
\end{eqnarray*}
where $A,B\in M_{n}^{+}$, $X\in M_{n}$, and
$t>0$. This series of inequalities can be improved as follows.
\begin{thm}
\label{thm:Zhan_improved}Let $A,B\in M_{n}^{+}$, $X\in M_{n}$,
and $N$ be a nonnegative integer. If $t>0$ and $0\leq v\leq1$, then
\begin{eqnarray*}
|||\,|A^{1-v}XB^{v}|^{t}|||\cdot|||\,|A^{v}XB^{1-v}|^{t}||| & \leq & |||\,|AX|^{t}|||\cdot|||\,|XB|^{t}|||\\
 &  & -\sum_{n=0}^{N-1}r_{n}(v)\sum_{k=1}^{2^{n}}\Delta_{f}(n,k)\chi_{(\frac{k-1}{2^{n}},\frac{k}{2^{n}})}(v),
\end{eqnarray*}
where $f(v)=|||\,|A^{1-v}XB^{v}|^{t}|||\cdot|||\,|A^{v}XB^{1-v}|^{t}|||$.\end{thm}
\begin{proof}
It follows from Theorem \ref{lem:convex_lemma} since  $f(v)=|||\,|A^{1-v}XB^{v}|^{t}|||\cdot|||\,|A^{v}XB^{1-v}|^{t}|||$
is convex on $[0,1]$ (see \cite[Theorem 4.30]{Zhan}).
\end{proof}
\noindent In particular, if $N=1$ the above theorem reduces to
\begin{eqnarray*}
|||\,|A^{1-v}XB^{v}|^{t}|||\cdot|||\,|A^{v}XB^{1-v}|^{t}||| & \leq & (1-2r_{0}(v))|||\,|AX|^{t}|||\cdot|||\,|XB|^{t}|||\\
 &  & +2r_{0}(v)|||\,|A^{1/2}XB^{1/2}|^{t}|||^{2},
\end{eqnarray*}
where $0\leq v\leq1$.

Similarly to the previous theorem, we can also utilize  convexity  of a function $f(v)=|||\,|A^{v}XB^{v}|^{t}|||\cdot|||\,|A^{-v}XB^{-v}|^{t}|||$
on the interval $[-1,1]$ (for more details, see \cite[Corollary 4.32]{Zhan}).
\begin{thm}
Let $A,B\in M_{n}^{+}$, $X\in M_{n}$,
and $N$ be a nonnegative integer. If $t>0$ and $-1\leq v\leq1$, then
\begin{eqnarray*}
|||\,|A^{v}XB^{v}|^{t}|||\cdot|||\,|A^{-v}XB^{-v}|^{t}||| & \leq & |||\,|AXB|^{t}|||\cdot|||\,|A^{-1}XB^{-1}|^{t}|||\\
 &  & -\sum_{n=0}^{N-1}s_{n}(v)\sum_{k=1-2^{n-1}}^{2^{n-1}}\Delta_{f}(n,k)\chi_{(\frac{k-1}{2^{n-1}},\frac{k}{2^{n-1}})}(v),
\end{eqnarray*}
where $f(v)=|||\,|A^{v}XB^{v}|^{t}|||\cdot|||\,|A^{-v}XB^{-v}|^{t}|||$
and $s_{n}(v)=r_{n}(\frac{v+1}{2})$.\end{thm}
\begin{proof}
Applying Theorem \ref{lem:convex_lemma} to $g(v)=f(2v-1)$, $0\leq v\leq1$, we have
\begin{equation}
f(2v-1)\leq(1-v)f(-1)+vf(1)-\sum_{n=0}^{N-1}r_{n}(v)\sum_{k=1}^{2^{n}}\Delta_{g}(n,k)\chi_{(\frac{k-1}{2^{n}},\frac{k}{2^{n}})}(v)\label{eq:g_to_f}.
\end{equation}
Now, since
\[
\Delta_{g}(n,k)=f(\frac{k-1}{2^{n-1}}-1)+f(\frac{k}{2^{n-1}}-1)-2f(\frac{2k-1}{2^{n}}-1),
\]
replacing $v$ by $\frac{v+1}{2}$ and $k$ by $2^{n-1}-k$ in (\ref{eq:g_to_f}), we obtain
\[
f(v)\leq\frac{1-v}{2}f(-1)+\frac{1+v}{2}f(1)-\sum_{n=0}^{N-1}s_{n}(v)\sum_{k=1-2^{n-1}}^{2^{n-1}}\Delta_{f}(n,k)\chi_{(\frac{k-1}{2^{n-1}},\frac{k}{2^{n-1}})}(v),
\]
which represents the desired inequality.
\end{proof}
\noindent In particular, if $N=1$ the above result reduces to
\begin{eqnarray*}
|||\,|A^{v}XB^{v}|^{t}|||\cdot|||\,|A^{-v}XB^{-v}|^{t}||| & \leq & (1-2s_{0}(v))|||\,|A^{-1}XB^{-1}|^{t}|||\cdot|||\,|AXB|^{t}|||\\
 &  & +2s_{0}(v)|||\,|X|^{t}|||^{2},
\end{eqnarray*}
where $-1\leq v\leq1$.

To conclude the paper, we will  improve the following inequality involving positive definite matrices and arithmetic mean (see \cite[pp. 554-555]{H_J}):
\begin{equation}
(A\nabla_{v}B)^{-1}\leq A^{-1}\nabla_{v}B^{-1},\label{eq:inverse_eq}
\end{equation}
where $A,B\in M_{n}^{+}$ and $0\leq v\leq1$. This inequality can also be refined by virtue of Theorem \ref{lem:convex_lemma}.
\begin{thm}
If $A,B\in M_{n}^{+}$ and $0\leq v\leq1$, then
\[
(A\nabla_{v}B)^{-1}\leq A^{-1}\nabla_{v}B^{-1}-\sum_{n=0}^{N-1}r_{n}(v)\sum_{k=1}^{2^{n}}F_{n,k}(A,B)\chi_{(\frac{k-1}{2^{n}},\frac{k}{2^{n}})}(v),
\]
where
\[
F_{n,k}(A,B)=(A\nabla_{\frac{k-1}{2^{n}}}B)^{-1}+(A\nabla_{\frac{k}{2^{n}}}B)^{-1}-2(A\nabla_{\frac{2k-1}{2^{n+1}}}B)^{-1}.
\]
\end{thm}
\begin{proof}
Let $c>0$. Applying Theorem \ref{lem:convex_lemma} to the convex function
$f(v)=(1-v+vc)^{-1}$, we have
\begin{eqnarray*}
(1-v+vc)^{-1} & \leq & 1-v+vc^{-1}-\sum_{n=0}^{N-1}r_{n}(v)\sum_{k=1}^{2^{n}}\Delta_{f}(n,k)\chi_{(\frac{k-1}{2^{n}},\frac{k}{2^{n}})}(v).
\end{eqnarray*}
Now, the result follows by the functional calculus as in Theorems \ref{prvi}, \ref{drugi}, and \ref{treci}.
\end{proof}
\noindent If $N=0$, the above theorem reduces to  inequality (\ref{eq:inverse_eq}), while for $N=1$ we obtain relation
\[
(A\nabla_{v}B)^{-1}\leq A^{-1}\nabla_{v}B^{-1}-2r_{0}(v)\left(A^{-1}\nabla B^{-1}-(A\nabla B)^{-1}\right),
\]
where $0\leq v\leq1$.


\begin{thebibliography}{10}
\bibitem{Bhatia}R. Bhatia, \textit{Matrix Analysis}, Springer-Verlag,
1997.



\bibitem{Choi_reverse}D. Choi, \textit{Multiple-term refinements of Young type inequalities}, preprint.

\bibitem{Dragomir}S. Dragomir, \textit{On new refinements and reverses
of Young's operator inequality}, http://arxiv.org/abs/1510.01314.

\bibitem{Furuichi}S. Furuichi, \textit{Refined Young inequalities
with Specht\textquoteright s ratio,} J. Egypt. Math. Soc. {\bf 20} (2012),
46--49.

\bibitem{H_J}R. Horn and C. Johnson, \textit{Topics in Matrix Analysis},
Cambridge U.P., New York, 1985.

\bibitem{Kittaneh}F. Kittaneh, \textit{On the convexity of the Heinz
means}, Integr. Equ. Oper. Theory {\bf 68} (2010), 519-- 527.

\bibitem{Kittaneh_Krnic}F. Kittaneh and M. Krni\' c, \textit{Refined
Heinz operator inequalities}, Linear and Multilinear Algebra {\bf 61} (2013), 1148--1157.

\bibitem{Kittaneh_Manasrah}F. Kittaneh and Y. Manasrah, \textit{Improved
Young and Heinz inequalities for matrices}, J. Math. Anal. Appl. {\bf 361}
(2010),  262--269.

\bibitem{Kitta_M2}F. Kittaneh and Y. Manasrah, \textit{Reverse Young
and Heinz inequalities for matrices}, Linear and Multilinear Algebra
{\bf 59} (2011),  1031--1037.

\bibitem{Liao_Wu}W. Liao and J. Wu, \textit{Improved Young and Heinz
inequalities with the Kantorovich constant}, J. Math. Ineq. {\bf 10}
(2016), 559--570.

\bibitem{Mincluete}N. Minculete, \textit{A refinement of the Kittaneh\textendash Manasrah
inequality}, Creat. Math. Inform. {\bf 20} (2011),  157--162.

\bibitem{Sabab_Choi_Young}M. Sababheh and D. Choi, \textit{A complete
refinement of Young's inequality,} J. Math. Anal. Appl. {\bf 440} (2016), 379--393.

\bibitem{Tominaga}M. Tominaga, \textit{Specht\textquoteright s ratio
in the Young inequality}, Sci. Math. Japon. {\bf55} (2002), 583--588.

\bibitem{Wu_Zhao}J. Wu and J. Zhao, \textit{Operator inequalities
and reverse inequalities related to the Kittaneh-Manasrah inequalities},
Linear and Multilinear Algebra. {\bf 62} (2014), 884--894.

\bibitem{Zhan}X. Zhan, \textit{Matrix inequalities}, Springer-Verlag,
Berlin (2002).

\bibitem{Zhao_Wu}J. Zhao and J. Wu, \textit{Operator inequalities
involving improved Young and its reverse inequalities}, J. Math. Anal.
Appl. {\bf 421} (2015),  1779--1789.\end{thebibliography}
\end{document}